\documentclass[12pt]{amsart}

\usepackage{amscd, amsmath, amssymb, amsthm, amsfonts, amsxtra, 
enumerate, latexsym, tabularx, setspace}
\usepackage[all]{xy}
\usepackage[dvips]{graphicx,color}

\def\FF{{\mathbb{F}}}

\def\R{{\mathbb{R}}}
\def\ZZ{{\mathbb{Z}}}
\def\Cl{{\mathrm{Cl}}}
\def\Coker{{\mathrm{Coker}}}
\def\Der{{\mathrm{Der}}}

\def\Hom{{\mathrm{Hom}}}

\def\Proj{{\mathrm{Proj\; }}}
\def\Spec{{\mathrm{Spec\; }}}

\theoremstyle{plain}
\newtheorem{thm}{Theorem}[section]

\newtheorem{lem}[thm]{Lemma}

\theoremstyle{definition}

\title{Quotients of smooth projective toric varieties by $\mu_p$ in positive characteristics $p$}
\author{Tadakazu Sawada}
\address{Department of General Education, National Institute of Technology, Fukushima College, 
30 Aza-Nagao, Kamiarakawa, Iwaki-shi, Fukushima 970-8034, Japan}
\email{sawada@fukushima-nct.ac.jp}


\begin{document}
\maketitle
\markboth{Tadakazu Sawada}{Quotients of smooth projective toric varieties by $\mu_p$ in 
positive characteristics $p$}

\begin{abstract}
In this paper we show that quotients of smooth projective toric varieties by $\mu_p$ in positive 
characteristics $p$ are toric varieties. 
\end{abstract}

\section*{Introduction}
We work over an algebraically closed field $k$ of positive characteristics $p$. A study of purely 
inseparable morphisms is one of a fundamental subject of algebraic geometry in positive characteristic. 
Let $X$ be a smooth variety over $k$. It is well known that if a morphism $f:X\rightarrow Y$ of 
varieties is purely inseparable of exponent one, then there exists a rational vector field 
$\delta \in H^0(X, T_X\otimes K(X))$, where $T_X$ is the tangent bundle of $X$ and $K(X)$ is the 
function field of $X$, such that $X/\delta \cong Y$. (See Rudakov and {\v{S}}afarevi{\v{c}} \cite{RS}.) 
It is also well known that if $X$ have a non-trivial $\mu_p$-action, then the quotient by $\mu_p$ 
exists and the quotient map $X\rightarrow X/\mu_p$ is purely inseparable of exponent one. (See 
Tziolas \cite{Tziolas}.) Tziolas shown the connection between actions of $\mu_p$ and global vector 
fields in \cite{Tziolas}. In this paper, we consider structures of $\mu_p$-quotients of smooth projective 
toric varieties by using the Tziolas' characterization. 

The organization of this paper is as follows. In section $1$, we fix notation. In section $2$, we review 
generalities on $p$-closed derivations quickly. In section $3$, we give explicit descriptions of global 
sections of the tangent bundle of smooth projective toric varieties. In section $4$, we consider 
structures of quotients of smooth projective toric varieties by $\mu_p$ in positive characteristics $p$. 

\section{Notation}
We follows the notation given in Hartshorne \cite{H} and Cox, Little and Schenck \cite{CLS}. Here we 
pick up some notations for toric varieties. 

\begin{tabular}{ll}
$N$ & lattice $\ZZ^n$\\
$N_{\R}$ & vector space $N\otimes_{\ZZ}\R$ \\
$M$ & dual lattice of $N$, equals $\Hom_{\ZZ}(N,\ZZ)$ \\
$M_{\R}$ & vector space $M\otimes_{\ZZ}\R$ \\
$\sigma$ & rational convex polyhedral cone in $N_{\R}$\\
$\Sigma$ & fan in $N_{\R}$ \\
$\Sigma (1)$ & rays of a fan $\Sigma$ in $N_{\R}$ \\
$\rho$ & ray in $N_{\R}$ \\
$X_{\Sigma}$ & toric variety of a fan $\Sigma$ in $N_{\R}$ \\
$K(X)$ & function field of a variety $X$ \\
$T_X$ & tangent bundle of a smooth variety $X$ \\
$\Cl(X)$ & divisor class group of a variety $X$ \\
$D_{\rho}$ & torus-invariant prime divisor on $X_{\Sigma}$ of a ray $\rho \in \Sigma (1)$ \\
$[D]$ & divisor class of a divisor $D$ \\
in $\Cl (X)$\\
$\FF_p$ & finite field $\ZZ/p\ZZ$\\
$\Der_k\, R$ & $k$-linear derivations on a $k$-algebra $R$ \\
$k\langle S\rangle$ & $k$-vector space spanned by a set $S$ 
\end{tabular}

\section{Generalities on $p$-closed derivations}
Let $R$ be a $k$-algebra and $D \in \Der_k\, R$. We define the constant ring of $D$, denoted $R^{D}$, 
by $R^{D}:=\{x\in R|D (x)=0\}$. 

Let $K$ be a field over $k$ and $D \in \Der_k\, K$. Since $x^p \in K^{D}$ for all $x\in K$, the field extension 
$K/K^{D}$ is algebraic of purely inseparable of exponent one. It is known that the minimal polynomial 
$\mu_{D}\in K^{D}[t]$ of the $K^{D}$-linear map $D : K\rightarrow K$ is expressed as 
$\mu_{D}=t^{p^e}+c_{e-1}t^{p^{e-1}}+\cdots +c_{1}t^p+c_0 t$ for some $c_{e-1},\ldots ,c_0 \in K^{D}$ and 
that $\deg \mu_{D} =[K:K^D]$. (See Aramova and Avramov \cite{AA} for more details.) If $D^p=\alpha D$ 
for some $\alpha \in K^{D}$, where $D^p$ is the $p$-times composition as a derivation, we say $D$ is 
$p$-closed. Suppose that $D$ is $p$-closed. Then $D^p-\alpha D=0$ for some $\alpha\in K^{D}$. Hence 
$\deg \mu_{D}=p$, so that the field extension $K/K^{D}$ is algebraic of purely inseparable of degree $p$. 

Let $X$ be a smooth variety over $k$ and $\delta \in H^0(X, T_X\otimes K(X))=\Der_k\, K(X)$ be a rational 
vector field. Let $\{U_i=\Spec R_i\}_i$ be an affine covering of $X$. The quotient $X/\delta$ is described by 
glueing $\{U_i/\delta=\Spec R_i^{\delta} \}_i$ and the quotient map $\pi:X\rightarrow X/\delta$ is induced by 
inclusions $R_i^{\delta}\subset R_i$. If $\delta$ is $p$-closed, the quotient map $\pi:X\rightarrow X/\delta$ is 
purely inseparable of degree $p$. 

\section{Global sections of a smooth projective toric variety}
Let $X_{\Sigma}$ be an $n$-dimensional smooth projective toric variety. Let $S=k[x_{\rho}|\rho \in \Sigma (1)]$ 
be the total coordinate ring of $X_{\Sigma}$. We define degree $0$ homomorphism 
$\alpha : \Cl (X_{\Sigma})\spcheck \otimes_{\ZZ} S \rightarrow \bigoplus_{\rho \in \Sigma (1) } S( [D_{\rho}] )$ 
of graded $S$-modules by $\varphi \otimes 1 \mapsto \sum_{\rho \in \Sigma (1) } \varphi ([D_{\rho}]) x_{\rho} e_{\rho}$, 
where $\Cl (X_{\Sigma})\spcheck=\Hom_{\ZZ}(\Cl (X_{\Sigma}),\ZZ)$ and $e_{\rho}$ are basis of 
$\bigoplus_{\rho \in \Sigma (1) } S( [D_{\rho}] )$ in degree $-[D_{\rho}]$. Then we have the exact sequence 
$$0\rightarrow \Cl (X_{\Sigma})\spcheck \otimes_{\ZZ} S 
\xrightarrow{\alpha} \bigoplus_{\rho \in \Sigma (1) } S( [D_{\rho}] )\rightarrow \Coker\, \alpha \rightarrow 0.$$
This give rise to the Euler sequence 
$$0\rightarrow \Cl (X_{\Sigma})\spcheck\otimes_{\ZZ} {\mathcal{O}}_{X_{\Sigma}} 
\rightarrow \bigoplus_{\rho \in \Sigma (1)} {\mathcal{O}}_{X_{\Sigma}}([D_{\rho}] )\rightarrow T_{X_{\Sigma}}\rightarrow 0$$
and there exists an isomorphism $\beta: (\Coker\, \alpha)\ \widetilde{} \rightarrow T_{X_{\Sigma}}$, where 
$(\Coker\,\alpha)\ \widetilde{}\ $ is the sheaf associated to $\Coker\,\alpha$. We denote by $V_{\rho}$ the 
$k$-vector space $H^0(X_{\Sigma}, \mathcal{O}_{X_{\Sigma}}([D_{\rho}]))=(S([D_{\rho}]))_0
=k\langle \prod x_{\rho}| [D_{\rho}]=[\mathrm{div}\,(\prod x_{\rho})]\rangle$ for $\rho \in \Sigma(1)$. 
We have isomorphisms of $k$-vector spaces 
\begin{eqnarray*}
H^0(X_{\Sigma}, T_{X_{\Sigma}})
&\cong & \left(\bigoplus_{\rho \in \Sigma (1)} V_{\rho}e_{\rho}\right)
\Big/k\left\langle  \sum_{\rho \in \Sigma (1)} \varphi ([D_{\rho}]) x_{\rho} e_{\rho} \Bigg| \varphi \in \Cl(X_{\Sigma})\spcheck\right\rangle \\
&\cong & \left(\bigoplus_{\rho \in \Sigma (1)} V_{\rho}\dfrac{\partial}{\partial x_{\rho}}\right)
\Big/k\left\langle  \sum_{\rho \in \Sigma (1)} \varphi ([D_{\rho}]) x_{\rho} \dfrac{\partial}{\partial x_{\rho}} \Bigg| \varphi \in \Cl(X_{\Sigma})\spcheck\right\rangle,
\end{eqnarray*}
where the first isomorphism is induced by $\beta$ and the second isomorphism is defined by sending 
$\overline{e_{\rho}}$ to $\overline{{\partial}/{\partial x_{\rho}}}$. 

Let $\sigma \in \Sigma$ be an $n$-dimensional cone and 
$y={(\prod x_{\alpha}^{a_{\alpha}})}/{(\prod x_{\beta}^{b_{\beta}})}\in k[\sigma\spcheck \cap M]$, where 
$\prod x_{\alpha}$ and $\prod x_{\beta}$ have no common factors, be one of the generators of 
$k[\sigma\spcheck \cap M]$ as a $k$-algebra. Let $\rho \in \Sigma (1)$ and $\prod x_{\xi}$ be an element 
of the total coordinate ring $S$ such that $[\mathrm{div}\, (\prod x_{\xi})]=[D_{\rho}]$. By considering 
actions to $y$, we think of $x_{\rho}{\partial}/{\partial x_{\rho}}$ and $(\prod x_{\xi})\,{\partial}/{\partial x_{\rho}}$ 
as derivations on $k[\sigma\spcheck \cap M]$: 
\begin{itemize}
\item If $x_{\rho}$ appears in $\prod x_{\alpha}$ (resp. $\prod x_{\beta}$), then 
$x_{\rho}\dfrac{\partial}{\partial x_{\rho}}y=a_{\rho} y$ (resp. $-b_{\rho} y$), so that 
$x_{\rho} {\partial}/{\partial x_{\rho}}= a_{\rho} y {\partial}/{\partial y}$ 
$\left({\rm resp.} -b_{\rho} y{\partial}/{\partial y}\right)$ on $k[\sigma\spcheck \cap M]$. 
\item Suppose that ${x_{\rho}}/{(\prod x_{\xi})}\not=y$. 
If $x_{\rho}$ appears in $\prod x_{\alpha}$ (resp. $\prod x_{\beta}$), then 
$(\prod x_{\xi})\dfrac{\partial}{\partial x_{\rho}}y=a_{\rho} \dfrac{\prod x_{\xi}}{x_{\rho}}y$ 
$\left({\rm resp.} -b_{\rho} \dfrac{\prod x_{\xi}}{x_{\rho}}y\right)$, so that 
$(\prod x_{\xi})\dfrac{\partial}{\partial x_{\rho}}=a_{\rho}\dfrac{\prod x_{\xi}}{x_{\rho}} y\dfrac{\partial}{\partial y}$ 
$\left({\rm resp.} -b_{\rho}\dfrac{\prod x_{\xi}}{x_{\rho}} y\dfrac{\partial}{\partial y}\right)$ on $k[\sigma\spcheck \cap M]$. 
\item If ${x_{\rho}}/{(\prod x_{\xi})}=y$, then $\prod x_{\xi}\dfrac{\partial}{\partial x_{\rho}}y=1$, so that 
$(\prod x_{\xi})\, {\partial}/{\partial x_{\rho}}={\partial}/{\partial y}$ on $k[\sigma\spcheck \cap M]$. 
\end{itemize}

Set  
$$V:=\left(\bigoplus_{\rho \in \Sigma (1)} V_{\rho}e_{\rho}\right)
\Big/k\left\langle  \sum_{\rho \in \Sigma (1)} \varphi ([D_{\rho}]) x_{\rho} e_{\rho} \Bigg| \varphi \in \Cl(X_{\Sigma})\spcheck\right\rangle$$
and 
$$\Der_k\, \mathcal{O}_{X_{\Sigma}}:=\left(\bigoplus_{\rho \in \Sigma (1)} V_{\rho}\dfrac{\partial}{\partial x_{\rho}}\right)
\Big/k\left\langle  \sum_{\rho \in \Sigma (1)} \varphi ([D_{\rho}]) x_{\rho} \dfrac{\partial}{\partial x_{\rho}} \Bigg| \varphi \in \Cl(X_{\Sigma})\spcheck\right\rangle.$$
We define a {\it restriction map} $\varphi_U:\Der_k\, \mathcal{O}_{X_{\Sigma}} \rightarrow H^0(U,T_X)$ 
for an open set $U\subset U_{\sigma}$ as follows: 
\begin{itemize}
\item If $x_{\rho}$ appears in $\prod x_{\alpha}$ (resp. $\prod x_{\beta}$), then 
$x_{\rho}{\partial}/{\partial x_{\rho}} \mapsto a_{\rho} y{\partial}/{\partial y}$ 
$\left({\rm resp.} -b_{\rho} y{\partial}/{\partial y}\right)$. 
\item Suppose that ${x_{\rho}}/{(\prod x_{\xi})}\not=y$. If $x_{\rho}$ appears in $\prod x_{\alpha}$ 
(resp. $\prod x_{\beta}$), then 
$(\prod x_{\xi})\dfrac{\partial}{\partial x_{\rho}}\mapsto a_{\rho}\dfrac{\prod x_{\xi}}{x_{\rho}} y\dfrac{\partial}{\partial y}$ 
$\left({\rm resp.} -b_{\rho}\dfrac{\prod x_{\xi}}{x_{\rho}} y\dfrac{\partial}{\partial y}\right)$. 
\item If ${x_{\rho}}/{(\prod x_{\xi})}=y$, then 
$(\prod x_{\xi})\, {\partial}/{\partial x_{\rho}} \mapsto {\partial}/{\partial y}$. 
\end{itemize}
Then we have a commutative diagram 
$$\xymatrix{
H^0(X_{\Sigma},T_{X_{\Sigma}}) \ar[d]_{\rho_{X_{\Sigma}U}} 
& V \ar[l]_(.4){\beta(X_{\Sigma})} \ar[r] \ar[d]_{{\rho'}_{X_{\Sigma}U}} 
& \Der_k\, \mathcal{O}_{X_{\Sigma}} \ar@/^22mm/[lld]^{\varphi_U}\\
H^0(U,T_{X_{\Sigma}}) & H^0(U, (\Coker\,\alpha)\ \widetilde{}\ ) \ar[l]_(.56){\beta(U)} & 
}$$
where $\rho_{X_{\Sigma}U}$ (resp. ${\rho'}_{X_{\Sigma}U}$) is the restriction map of $T_{X_{\Sigma}}$ 
(resp. $(\Coker\,\alpha)\ \widetilde{}\ $). Therefore we have the following result. 

\begin{lem}\label{1}
For any open set $U\subset U_{\sigma}$, we have a commutative diagram
$$\xymatrix{
H^0(X_{\Sigma},T_{X_{\Sigma}}) \ar[r] \ar[d]_{\rho_{XU}} 
& \Der_k\, \mathcal{O}_{X_{\Sigma}} \ar[ld]^{\varphi_U}\\
H^0(U ,T_{X_{\Sigma}}) & 
}$$
\end{lem}

\begin{lem}\label{2}
Let $X_{\Sigma}$ be a smooth project toric variety and $S$ be the total coordinate ring of $X_{\Sigma}$. 
Let $\delta \in H^0(X_{\Sigma},T_{X_{\Sigma}})\ (\subset \Der_k K(X_{\Sigma}))$ and $\overline{D}$ be the 
corresponding element of $\Der_k\,\mathcal{O}_{X_{\Sigma}}$, where 
$D\in \bigoplus_{\rho \in \Sigma (1)} V_{\rho}\hspace{0.5mm} \partial/\partial x_{\rho}$. Then 
$X_{\Sigma}/\delta \cong \Proj S^D$. 
\end{lem}
\begin{proof}
Let $F\in S$ be a homogeneous element of the irrelevant ideal of $S$ such that $F\in S^D$. We denote by 
$D^{X_{\Sigma}}_{+}(F)$ (resp. $D^{\Proj S^D}_{+}(F)$) the open set of $X_{\Sigma}$ (resp. $\Proj S^D$) defined 
by $F$. Let $\sigma \in \Sigma$ be an $n$-dimensional cone. Multiplying $F$ by 
$\left(\prod_{\rho \not\in \sigma} x_{\rho}\right)^p$, we may assume that $D^{X_{\Sigma}}_{+}(F)\subset U_{\sigma}$. 
Since $(S^D)_{(F)}=(S_{(F)})^D$, we have $D^{\Proj S^D}_{+}(F)\cong \Spec (S^D)_{(F)}\cong \Spec (S_{(F)})^D$. 
On the other hand, we have $D^{X_{\Sigma}}_{+}(F)/\delta \cong \Spec (S_{(F)})^{\delta}= \Spec (S_{(F)})^D$ 
by Lemma~\ref{1}. Since $X_{\Sigma}/\delta$ is constructed by glueing $D^{X_{\Sigma}}_{+}(F)/\delta$, we have 
$X_{\Sigma}/\delta \cong \Proj S^D$. 
\end{proof}

\begin{lem}\label{3}
With the same hypotheses as Lemma 2.2, suppose that $\delta^p=\delta$. Then there exists 
$D'=\sum_{\rho\in \Sigma (1)}\alpha_{\rho}x_{\rho} \partial/\partial x_{\rho}\in \bigoplus_{\rho \in \Sigma (1)} V_{\rho}\hspace{0.5mm} \partial/\partial x_{\rho}$, 
where $\alpha_{\rho} \in \FF_p$, such that $S^D\cong S^{D'}$ as graded rings. 
\end{lem}
\begin{proof}
In what follows we denote $k(x_{\rho}|\rho \in \Sigma (1))$ by $k(x_{\rho})$. There exist the following field extensions: 
$$\xymatrix{
 & k(x_{\rho}) \ar@{-}[d] \\
K(X_{\Sigma}) \ar@{-}[ru] \ar@{-}[d] & k(x_{\rho})^D  \\
K(X_{\Sigma})^D \ar@{-}[ru] &  
}$$
The field extension $k(x_{\rho})/k(x_{\rho})^D$ is algebraic of purely inseparable of exponent one. Since 
$K(X_{\Sigma})^D=K(X_{\Sigma})^{\delta}$ and $\delta$ is $p$-closed, the field extension $K(X_{\Sigma})/K(X_{\Sigma})^D$ 
is algebraic of purely inseparable of degree $p$. The dehomogenized elements of the generators of 
$K(X_{\Sigma})$ as a $K(X_{\Sigma})^D$-vector space generate $k(x_{\rho})$ as a $k(x_{\rho})^D$-vector space. Hence 
we have $[k(x_{\rho}):k(x_{\rho})^D]=p$. Since the degree of the minimal polynomial of the $k(x_{\rho})^D$-linear map 
$D:k(x_{\rho})\rightarrow k(x_{\rho})$ equals $p$, $D$ is $p$-closed. 

Let $D^p=\alpha D$, where $\alpha \in k(x_{\rho})^D$. Since 
$D\in \bigoplus_{\rho \in \Sigma (1)} V_{\rho}\hspace{0.5mm} \partial/\partial x_{\rho}$, we see that $\alpha \in k$. 
Let $\beta\in k$ be a root of the equation $\alpha x^p-x=0$. Since 
$(\beta D)^p=\beta^pD^p=\beta^p\alpha D=\beta D$, we may assume that $D^p=D$ by replacing $D$ by $\beta D$. 
$D$ induces a $k$-linear map $\psi : \bigoplus_{\rho \in \Sigma (1)} V_{\rho}\rightarrow \bigoplus_{\rho \in \Sigma (1)} V_{\rho}$. 
Since $D^p=D$, the minimal polynomial of $\psi\in k[t]$ divides $t^p-t=t(t-1)\cdots (t-(p-1))$. Hence $D$ is diagonalizable. 
Let $y_{\rho}$ be an eigenvector of $\psi$ with an eigenvalue $a_{\rho} \in \FF_p$ for $\rho \in \Sigma (1)$. By coordinate 
changes $x_{\rho} \mapsto y_{\rho}$ for $\rho \in \Sigma (1)$, we have 
$D=\sum_{\rho\in \Sigma (1)}a_{\rho}y_{\rho}\partial/\partial y_{\rho}\in \bigoplus_{\rho \in \Sigma (1)} V_{\rho}\hspace{0.5mm} \partial/\partial y_{\rho}$. 
Therefore the proof is completed. 

Note that $D$ is decomposed as $D=\bigoplus_{\rho\in \Sigma (1)} D_{\rho}$, where $D_{\rho}\in V_{\rho}\hspace{0.5mm} \partial/\partial x_{\rho}$, 
and $D_{\rho}$ defines $k$-linear map $V_{\rho}\rightarrow V_{\rho}$ for each $\rho \in \Sigma(1)$. Hence the coordinate 
changes $x_{\rho}\mapsto y_{\rho}$ induce the (well-defined) graded automorphism on $S$ and define an isomorphism on 
$X_{\Sigma}$. 
\end{proof}

\begin{lem}\label{4}
Let $X_{\Sigma}$ be a smooth projective toric variety and $S$ be the total coordinate ring of $X_{\Sigma}$. Let 
$D=\sum_{\rho\in \Sigma (1)}a_{\rho}x_{\rho} {\partial}/{\partial x_{\rho}}\in \bigoplus_{\rho \in \Sigma (1)} V_{\rho}\hspace{0.5mm} \partial/\partial x_{\rho}$, 
where $a_{\rho}\in \FF_p$. Then $\Proj S^D$ is a toric variety. 
\end{lem}
\begin{proof}
Let $\sigma \in \Sigma$ be an $n$-dimensional cone and $\rho_1,\ldots ,\rho_n \in \Sigma (1)$ be the rays which span $\sigma$. 
In what follows we consider a coordinate $n$-tuple of an element of $N$ with respect to $\ZZ$-basis $\rho_1,\ldots, \rho_n$. Let 
$U_{\sigma}=\Spec k[z_1,\ldots, z_n]$, where $z_i$ corresponds to $x_{\rho_i}$ respectively. Suppose that $\overline{D}$ is locally 
expressed as $\phi_{U_{\sigma}}(\overline{D})=\sum \alpha_i z_i{\partial}/{\partial z_i}$ with $\alpha_i \in \FF_p$. Then we see that 
$X_{\Sigma}/\overline{D}$ is the toric variety whose corresponding fan is $\Sigma$ in $N'\otimes \R$, where 
$N'=N+\ZZ\dfrac{1}{p}(\alpha_1,\ldots ,\alpha_n)$. ($N'$ is an overlattice of $N$.) Since $X_{\Sigma}/\overline{D}\cong \Proj S^D$, 
$\Proj S^D$ is a toric variety. 
\end{proof}

Let $X$ be a variety over $k$ and $\mu_p=\Spec k[x]/(x^p-1)$. Suppose that $X$ has a $\mu_p$-action. Tziolas shown that there 
exists a global vector field $\delta\in H^0(X,T_X)$ such that $\delta^p=\delta$ and $X/\mu_p \cong X/\delta$. (See Tziolas~\cite{Tziolas}.)

\begin{thm}
Let $X_{\Sigma}$ be a smooth projective toric variety and have a $\mu_p$-action. Then the quotient $X_{\Sigma}/\mu_p$ is a toric variety. 
\end{thm}
\begin{proof}
There exists a global vector field $\delta\in H^0(X_{\Sigma}, T_{X_{\Sigma}})$ such that $\delta^p=\delta$ and $X/\mu_p\cong X/\delta$. 
Let $S=k[x_{\rho}|\rho \in \Sigma (1)]$ be the total coordinate ring of $X_{\Sigma}$. Let 
$D\in \bigoplus_{\rho \in \Sigma (1)} V_{\rho}\hspace{0.5mm} \partial/\partial x_{\rho}$ be a derivation such that $\delta=\overline{D}$. 
By Lemma~\ref{3} there exists $D'=\sum_{\rho}a_{\rho}x_{\rho}\hspace{0.5mm}\partial/\partial x_{\rho}
\in \bigoplus_{\rho \in \Sigma (1)} V_{\rho}\hspace{0.5mm} \partial/\partial x_{\rho}$, where $a_{\rho}\in \FF_p$, such that 
$S^{D}\cong S^{D'}$ as graded rings. We have $X/\mu_p\cong X/\delta\cong \Proj S^D\cong \Proj S^{D'}$ by Lemma~\ref{2}. 
Therefore $X/\mu_p$ is a toric variety by Lemma~\ref{4}. 
\end{proof}

\providecommand{\bysame}{\leavevmode\hbox to3em{\hrulefill}\thinspace}
\providecommand{\MR}{\relax\ifhmode\unskip\space\fi MR }
\providecommand{\MRhref}[2]{%
  \href{http://www.ams.org/mathscinet-getitem?mr=#1}{#2}
}
\providecommand{\href}[2]{#2}


\begin{thebibliography}{1}

\bibitem{AA}
Annetta~G. Aramova and Luchezar~L. Avramov, \emph{Singularities of quotients by
  vector fields in characteristic {$p$}}, Math. Ann. \textbf{273} (1986),
  no.~4, 629--645. \MR{826462}

\bibitem{CLS}
David~A. Cox, John~B. Little, and Henry~K. Schenck, \emph{Toric varieties},
  Graduate Studies in Mathematics, vol. 124, American Mathematical Society,
  Providence, RI, 2011. \MR{2810322}

\bibitem{H}
Robin Hartshorne, \emph{Algebraic geometry}, Springer-Verlag, New
  York-Heidelberg, 1977, Graduate Texts in Mathematics, No. 52. \MR{0463157}

\bibitem{RS}
A.~N. Rudakov and I.~R. {\v{S}}afarevi{\v{c}}, \emph{Inseparable morphisms of
  algebraic surfaces}, Izv. Akad. Nauk SSSR Ser. Mat. \textbf{40} (1976),
  no.~6, 1269--1307, 1439. \MR{0460344}

\bibitem{Tziolas}
Nikolaos Tziolas, \emph{Quotients of schemes by {$\alpha_p$} or {$\mu_p$}
  actions in characteristic {$p>0$}}, Manuscripta Math. \textbf{152} (2017),
  no.~1-2, 247--279. \MR{3595379}

\end{thebibliography}
\end{document}